\documentclass{article}
\usepackage[utf8]{inputenc}
\pdfoutput=1
\usepackage[sc,osf]{mathpazo}
\usepackage[height=8.5in,width=6.2in,letterpaper]{geometry}
\usepackage[sort&compress,numbers]{natbib}
\usepackage[colorlinks=true,citecolor=blue,breaklinks]{hyperref}
\usepackage[hyphenbreaks]{breakurl} 
\usepackage[tracking]{microtype}

\makeatletter
\newlength\aftertitskip     \newlength\beforetitskip
\newlength\interauthorskip  \newlength\aftermaketitskip

\def\maketitle{\par
 \begingroup
   \def\thefootnote{\fnsymbol{footnote}}
   \def\@makefnmark{\hbox to 4pt{$^{\@thefnmark}$\hss}}
   \@maketitle \@thanks
 \endgroup
\setcounter{footnote}{0}
 \let\maketitle\relax \let\@maketitle\relax
 \gdef\@thanks{}\gdef\@author{}\gdef\@title{}\let\thanks\relax}

\def\@startauthor{\noindent \normalsize\bf}
\def\@endauthor{}
\def\@starteditor{\noindent \small {\bf Editor:~}}
\def\@endeditor{\normalsize}
\def\@maketitle{\vbox{\hsize\textwidth
 \linewidth\hsize \vskip \beforetitskip
 {\begin{center} \LARGE\@title \par \end{center}} \vskip \aftertitskip
 {\def\and{\unskip\enspace{\rm and}\enspace}%
  \def\addr{\small}%
  \def\email{\hfill\small\sf}%
  \def\name{\normalsize\bf}%
  \def\AND{\@endauthor\rm\hss \vskip \interauthorskip \@startauthor}
  \@startauthor \@author \@endauthor}
}}

\makeatother

\pdfoutput=1                    

\usepackage{floatrow} 
\newfloatcommand{capbtabbox}{table}[][\FBwidth]

\usepackage{graphicx}
\usepackage{amsmath,amsthm,amssymb,bm} 
\usepackage{amsfonts}
\usepackage{mathrsfs}
\usepackage{subfigure}
\usepackage{xspace}
\usepackage{array}
\usepackage{enumerate}
\usepackage{algorithm}
\usepackage{algorithmic}
\usepackage{appendix}
\usepackage{wrapfig}
\numberwithin{equation}{section}

\newtheorem{theorem}{Theorem}

\newtheorem{corr}[theorem]{Corollary}

\theoremstyle{definition}

\newtheorem{example}{Example}

\newcommand{\R}{\mathbb{R}}
\newcommand{\pp}{\mathbb{P}}

\newcommand{\pfrac}[2]{\left(\tfrac{#1}{#2}\right)}
\newcommand{\half}{\tfrac{1}{2}}
\newcommand{\gm}{{\#}}
\newcommand{\Gc}{\mathcal{G}}

\DeclareMathOperator*{\argmin}{argmin}
\DeclareMathOperator{\mydet}{det}
\DeclareMathOperator{\rank}{rank}
\newcommand{\algo}{\textsc{Yamsr}\xspace}
\newcommand{\pn}{\textsc{Pn}\xspace}
\newcommand{\norm}[1]{\Vert#1\Vert}

\begin{document}
\title{On the matrix square root via geometric optimization}
\author{\name Suvrit Sra \email suvrit@mit.edu\\ 
\addr Laboratory for Information and Decision Systems\\ Massachusetts Institute of Technology\\ Cambridge, MA, 02139, United States}

\maketitle

\vskip0.4cm
\hrule
\vskip0.4cm

\begin{abstract}
  This paper is triggered by the preprint ``\emph{Computing Matrix Squareroot via Non Convex Local Search}'' by Jain et al.~(\textit{\textcolor{blue}{arXiv:1507.05854}}), which analyzes gradient-descent for computing the square root of a positive definite matrix. Contrary to claims of~\citet{jain2015}, our experiments reveal that Newton-like methods compute matrix square roots rapidly and reliably, even for highly ill-conditioned matrices and without requiring commutativity. We observe that gradient-descent converges very slowly primarily due to tiny step-sizes and ill-conditioning. We derive an alternative first-order method based on geodesic convexity: our method admits a transparent convergence analysis ($< 1$ page), attains linear rate, and displays reliable convergence even for rank deficient problems. Though superior to gradient-descent, ultimately our method is also outperformed by a well-known scaled Newton method. Nevertheless, the primary value of our work is its conceptual value: it shows that for deriving gradient based methods for the matrix square root, \emph{the manifold geometric view of positive definite matrices can be much more advantageous than the Euclidean view}. 
\end{abstract}

\section{Introduction}
Matrix functions such as $A^\alpha$, $\log A$, or $e^A$ ($A \in \mathbb{C}^{n\times n}$) arise in diverse contexts. \citet{higham08} provides an engaging overview on the topic of matrix functions. A building block for efficient numerical evaluation of various matrix functions is the \emph{matrix square root} $A^{1/2}$, a function whose computation has witnessed great interest in the literature~\citep{hale08,ianna11,yuji,high86,ando81,anMoTr83}; see also~\citep[Ch.~6]{higham08}.

The present paper revisits the problem of computing the matrix square root for a symmetric positive semidefinite (psd) matrix.\footnote{The same ideas extend trivially to the Hermitian psd case.} Our work is triggered by the recent note of~\citep{jain2015}, who analyze a simple gradient-descent procedure for computing $A^{1/2}$ for a given psd matrix $A$. \citet{jain2015} motivate their work by citing weaknesses of a direct application of Newton's method (Algorithm~X in~\citep{high86}), which is known to suffer from ill-conditioning and whose convergence analysis depends on commutativity of its iterates, and thus breaks down as round-off error invalidates commutativity.

However, it turns out that the modified ``polar-Newton'' (\pn) method of~\citep[Alg.~6.21]{higham08} avoids these drawbacks, and offers a theoretically sound method with excellent practical behavior. At the same time, gradient-descent of~\citep{jain2015}, while initially appealing due to its simplicity, turns out to be empirically weak: it converges extremely slowly, especially on ill-conditioned matrices because its stepsizes become too small to ensure progress on computers with limited floating point computation.

At the expense of introducing ``yet another matrix square root'' (\algo) method, we present a non-Euclidean first-order method grounded in the geometric optimization framework of~\citep{sraHo15}. Our method admits a transparent convergence analysis, attains linear (geometric) convergence rate, and displays reliable behavior even for highly ill-conditioned problems (including rank-deficient problems). Although numerically superior to the gradient-descent approach of~\citep{jain2015}, like other linearly convergent strategies \algo is still outperformed by the well-known polar-Newton (scaled Newton) method.

In light of these observations, there are two key messages delivered by our paper:
\begin{list}{$\blacktriangleright$}{\leftmargin=1em}
\setlength{\itemsep}{0pt}
\item None of the currently known (to us) first-order methods are competitive with the established second-order polar-Newton algorithm for computing matrix square roots.\footnote{Presumably, the same conclusion also holds for many other matrix functions such as powers and logarithms.}
\item The actual value of this paper is conceptual: it shows that for deriving gradient based methods that compute matrix square roots, \emph{the differential geometric view of positive definite matrices is much more advantageous than the Euclidean view}. 
\end{list}
As a further attestation to the latter claim, we note that our analysis yields as a byproduct a convergence proof of a (Euclidean) iteration due to \citet{ando81}, which previously required a more intricate analysis. Moreover, we also obtain a geometric rate of convergence. 

Finally, we note that the above conceptual view underlies several other applications too, where the geometric view of positive definite matrices has also proved quite useful, e.g.,~\citep{cherianSra14,sraHo15,ssdiv,hoSra15b,wie12,wie12b,zhang12}.

\vspace*{-5pt}
\subsection{Summary of existing methods}
To place our statements in wider perspective we cite below several methods for matrix square roots. We refer the reader to~\citep[Chapters 6, 8]{higham08} for more details and many more references, as well as a broader historical perspective. For a condensed summary, we refer the reader to the survey of~\citet{ianna11} which covers the slightly more general problem of computing the matrix geometric mean.

\begin{enumerate}
\setlength{\itemsep}{-2pt}
\item \textbf{Eigenvectors}: Compute the decomposition $A=U\Lambda U^*$ and obtain $A^{1/2}=U\Lambda^{1/2}U^*$. While tempting, this can be slower for large and ill-conditioned matrices. In our case, since $A\succ 0$, this method coincides with the Schur-Cholesky method mentioned in~\citep{ianna11}.
\item \textbf{Matrix averaging}: Many methods exist that obtain the matrix geometric mean as a limit of simpler averaging procedures. These include scaled averaging iterations (which is essentially a matrix Harmonic-Arithmetic mean iteration that converges to the geometric mean) and their variants~\citep[\S5.1]{ianna11}. In spirit, the algorithm that we propose in this paper also falls in this category.
\item \textbf{Newton-like}: The classic Newton iteration of~\citep{high86}, whose weaknesses motivate the gradient-descent procedure of~\citep{jain2015}. However, a better way to implement this iteration is to use polar decomposition. Here, first we compute the Cholesky factorization $A=R^TR$, then obtain the square root as $A^{1/2}=UR$, where $U$ is the unitary polar factor of $R^{-1}$. This polar factor could be computed using the SVD of $R$, but the polar-Newton (\pn) iteration~\citep[Alg.~6.21]{higham08} avoids that and instead uses the \emph{scaled Newton method}
  \begin{equation*}
    U_{k+1} = \half\left(\mu_kU_k + \mu_k^{-1}U_k^{-T}\right),\qquad U_0 = R.
  \end{equation*}
  This iteration has the benefit that the number of steps needed to attain a desired tolerance can be predicted in advance, and it can be implemented to be backward stable. It turns out that empirically this iteration is works reliably even for extremely ill-conditioned matrices, and converges (locally) superlinearly. The scaled Halley-iteration of~\citet{yuji} converges even more rapidly and can be implemented in a numerically stable manner~\citep{naHi12}.
\item \textbf{Quadrature}: Methods based on Gaussian and Gauss-Chebyshev quadrature and rational minimax approximations to $z^{-1/2}$ are surveyed in~\citep[\S5.4]{ianna11}. Among others, these methods easily tackle computation of general powers, logarithms, and some other matrix functions~\citep{ianMe,hale08}. 
\item \textbf{First-order methods}: The binomial iteration which uses $(I-C)^{1/2}=I-\sum_j \alpha_j C^j$ for suitable $\alpha_j > 0$ and applies when $\rho(C)<1$~\citep[\S6.8.1]{higham08}. The gradient-descent method of~\citep{jain2015} which minimizes $\|X^2-A\|_F^2$. Like the binomial iteration, this method does not depend on linear system solves (or matrix inversions) and uses only matrix multiplication. The method introduced in section~\ref{sec:algo} is also a first-order method, though cast in a non-Euclidean space; it does, however, require (Cholesky based) linear system solves. 
\end{enumerate}
We mention other related work in passing: fast computation of a matrix $C$ such that $A^p \approx CC^T$ for SDD matrices~\citep{chengCheng}; Krylov-subspace methods and randomized algorithms for computing the product $f(A)b$~\citep[Chapter~13]{higham08}; and the broader topic of multivariate matrix means~\citep{andoLiMa04,mmtoolbox,ssdiv,lawlim08,nieBha13,jeuVaVa}.


\vspace*{-8pt}
\section{Geometric Optimization}
\label{sec:sdiv}
\vspace*{-8pt}
We present details of computing the matrix square root using geometric optimization. Specifically, we cast the problem of computing matrix square roots into the nonconvex optimization problem
\begin{equation}
  \label{eq:3}
  \min_{X \succ 0}\quad \delta_S^2(X, A) + \delta_S^2(X, I),
\end{equation}
whose unique solution is the desired matrix square root $X^* = A^{1/2}$; here $\delta_S^2$ denotes the \emph{S-Divergence}~\citep{ssdiv}:
\begin{equation}
  \label{eq:1}
  \delta_S^2(X, Y) := \log\det\pfrac{X+Y}{2} - \half\log\det(XY).
\end{equation}
To present our algorithm for solving~\eqref{eq:3} let us first recall some background.

The crucial property that helps us solve~\eqref{eq:3} is \emph{geodesic convexity} of $\delta_S^2$, that is, convexity along geodesics in $\pp^n$. Consider thus, the geodesic (see \citep[Ch.~6]{bhatia07} for a proof that this is \emph{the} geodesic):
\begin{equation}
  \label{eq:4}
  X \gm_t Y := X^{1/2}(X^{-1/2}YX^{-1/2})^t X^{1/2},\quad X, Y \succ 0,\ \ t \in [0,1],
\end{equation}
that joins $X$ to $Y$. A function $f: \mathbb{P}^n \to \R$ is called \emph{geodesically convex} (g-convex) if it satisfies
\begin{equation}
  \label{eq:5}
  f(X \gm_t Y) \le (1-t) f(X) + t f(Y).
\end{equation}
\begin{theorem}
  \label{thm:gc}
  The S-Divergence~\eqref{eq:1} is jointly geodesically convex on (Hermitian) psd matrices.
\end{theorem}
\begin{proof}
  See~\citep[Theorem 4.4]{ssdiv}; we include a proof in the appendix for the reader's convenience.
\end{proof}
Similar to Euclidean convexity, g-convexity bestows the crucial property: ``local $\implies$ global.'' Consequently, we may use any algorithm that ensures local optimality; g-convexity will imply global optimality. We present one such algorithm in section~\ref{sec:algo} below.

But first let us see why solving~\eqref{eq:3} yields the desired matrix square root.
\begin{theorem}
  \label{thm.gm}
  Let $A, B \succ 0$. Then, 
  \begin{equation}
    \label{eq:6}
    A \gm_{1/2} B = \argmin\nolimits_{X \succ 0}\quad \delta_S^2(X, A) + \delta_S^2(X, B).
  \end{equation}
  Moreover, $A\gm_{1/2} B$ is equidistant from $A$ and $B$, i.e., $\delta_S^2(A, A\gm_{1/2} B) = \delta_S^2(B, A\gm_{1/2} B)$.
\end{theorem}
\begin{proof}
  See~\citep[Theorem 4.1]{ssdiv}. We include a proof below for convenience.

  Since the constraint set is open, we can simply differentiate the objective in~\eqref{eq:6}, and set the derivative to zero to obtain the necessary condition
  \begin{align*}
    \nonumber
    &\left(\tfrac{X+A}{2}\right)^{-1}\tfrac{1}{2} + \left(\tfrac{X+B}{2}\right)^{-1}\tfrac{1}{2} - X^{-1} = 0,\\
    \label{eq.33}
    &\implies\qquad X^{-1} = (X+A)^{-1}+(X+B)^{-1}\\
    \nonumber
    &\implies\qquad (X+A)X^{-1}(X+B) = 2X + A + B\\
    \nonumber
    &\implies\qquad B = XA^{-1}X.
  \end{align*}
  The latter is a Riccati equation whose \emph{unique} positive solution is $X=A\gm_{1/2}B$---see~\citep[Prop~1.2.13]{bhatia07}\footnote{There is a typo in the cited result, in that it uses $A$ and $A^{1/2}$ where it should use $A^{-1}$ and $A^{-1/2}$.}. Global optimality of this solution follows easily since $\delta_S^2$ is g-convex as per Thm.~\ref{thm:gc}.
\end{proof}

\begin{corr}
  The unique solution to~\eqref{eq:3} is $X^* = A\gm_{1/2} I = A^{1/2}$.
\end{corr}

\subsection{Algorithm for matrix square root}
\label{sec:algo}
We present now an iteration for solving~\eqref{eq:3}. As for Thm.~\ref{thm.gm}, we obtain here the optimality condition
\begin{equation}
  \label{eq:2}
  X^{-1} = (X+A)^{-1} + (X+I)^{-1}.
\end{equation}
Our algorithm solves~\eqref{eq:2} simply by running\footnote{Observe that the same algorithm also computes $A\gm_{1/2}B$ if we replace $I$ by $B$ in~\eqref{eq:7} and~\eqref{eq:8}.}
\begin{center}
  \fbox{
    \begin{minipage}[c]{0.9\linewidth}
      \vspace*{-4pt}
      \begin{align}
        \label{eq:7}
        X_0 &= \half(A+I)\\
        \label{eq:8}
        X_{k+1} &\gets [(X_k+A)^{-1} + (X_k+I)^{-1}]^{-1},\qquad k=0,1,\ldots.
      \end{align}
    \end{minipage}
  }
\end{center}
The following facts about our algorithm and its analysis are noteworthy:
\begin{enumerate}
  \setlength{\itemsep}{0pt}
\item Towards completion of this article we discovered that iteration~\eqref{eq:8} has been known in matrix analysis since over three decades! Indeed, motivated by electrical resistance networks, \citet{ando81} analyzed exactly the same iteration as~\eqref{eq:8} in 1981. Our convergence proof (Theorem~\ref{thm:convg}) is much shorter and more transparent---it not only reveals the geometry behind its convergence but also yields explicit bounds on the convergence rate, thus offering a new understanding of the classical work of~\citet{ando81}.
\item Initialization~\eqref{eq:7} is just one of many. Any matrix in the psd interval $[2(I + A^{-1})^{-1}, \half(A+I)]$ is equally valid; different choices can lead to better empirical convergence. 
\item Each iteration of the algorithm involves three matrix inverses. Theoretically, this costs  ``just'' $O(n^\omega)$ operations. In practice, we compute~\eqref{eq:8} using linear system solves; in \textsc{Matlab} notation:
  \begin{center}
    \verb$R = (X+A)\I + (X+I)\I; X = R\I;$
  \end{center}
\item The gradient-descent method of~\citep{jain2015} and the binomial method~\citep{higham08} do not require solving linear systems, and rely purely on matrix multiplication. But both turn out to be slower than~\eqref{eq:8} while also being more sensitive to ill-conditioning.
\item For ill-conditioned matrices, it is better to iterate~\eqref{eq:8} with $\alpha I$, for a suitable scalar $\alpha>0$; the final solution is recovered by dowscaling by $\alpha^{-1/2}$. A heuristic choice is $\alpha = \text{tr}(A)/\sqrt{n}$, which seems to work well in practice (for well-conditioned matrices $\alpha=1$ is preferable).
\end{enumerate}

\begin{theorem}[Convergence]
  \label{thm:convg}
  Let $\{X_k\}_{k\ge 0}$ be generated by (\ref{eq:7})-(\ref{eq:8}). Let $X^*=A^{1/2}$ be the optimal solution to~\eqref{eq:3}. Then, there exists a constant $\gamma < 1$ such that $\delta_S^2(X_k,X^*) \le \gamma^k\delta_S^2(X_0,X^*)$. Moreover, $\lim X_k=X^*$.
\end{theorem}
\begin{proof}
  Our proof is a specialization of the fixed-point theory in~\citep{leeLim,sraHo15}. Specifically, we prove that~\eqref{eq:8} is a fixed-point iteration under the \emph{Thompson part metric}
  \begin{equation}
    \label{eq:10}
    \delta_T(X,Y) := \norm{\log(X^{-1/2}YX^{-1/2})},
  \end{equation}
  where $\norm{\cdot}$ is the usual operator norm. The metric~\eqref{eq:10} satisfies many remarkable properties; for our purpose, we need the following three well-known properties (see e.g.,~\citep[Prop.~4.2]{sraHo15}):
  \begin{equation}
    \label{eq:11}
    \begin{split}
      \delta_T(X^{-1},Y^{-1}) = \delta_T(X,Y),\quad &\delta_T(X+A,Y+B) \le \max\{\delta_T(X,Y),\delta_T(A,B)\},\\
      \delta_T(X+A, Y+A) &\le \tfrac{\alpha}{\alpha+\lambda_{\min}(A)}\delta_T(X,Y),
      \quad \alpha=\max\{\norm{X},\norm{Y}\}.
    \end{split}
  \end{equation}
  Consider now the nonlinear map
  \begin{equation*}
    \Gc \equiv X \mapsto [(X+A)^{-1} + (X+I)^{-1}]^{-1},
  \end{equation*}
  corresponding to iteration~\eqref{eq:8}. Using properties~\eqref{eq:11} of the Thompson metric $\delta_T$ we have
  \begin{align*}
    \delta_T(\Gc(X), \Gc(Y)) &= \delta_T([(X+A)^{-1} + (X+I)^{-1}]^{-1}, [(Y+A)^{-1} + (Y+I)^{-1}]^{-1})\\
    &= \delta_T((X+A)^{-1} + (X+I)^{-1},(Y+A)^{-1} + (Y+I)^{-1})\\
    &\le \max\bigl\{\delta_T((X+A)^{-1}, (Y+A)^{-1}), \delta_T((X+I)^{-1}, (Y+I)^{-1})\bigr\}\\
    &=\max\left\{\delta_T(X+A, Y+A), \delta_T(X+I, Y+I)\right\}\\
    &\le\max\left\{\gamma_1\delta_T(X,Y), \gamma_2\delta_T(X,Y)\right\},\quad\gamma_1=\tfrac{\alpha}{\alpha+\lambda_{\min}(A)}, \gamma_2=\tfrac{\alpha}{\alpha+1},\ \alpha=\max\{\norm{X},\norm{Y}\},\\
    &\le \gamma\delta_T(X,Y),\qquad \gamma < 1,
  \end{align*}
  where we can choose $\gamma$ to be independent of $X$ and $Y$. Thus, the map $\Gc$ is a strict contraction. Hence, from the Banach contraction theorem it follows that $\delta_T(X_k,X^*)$ converges at a linear rate given by $\gamma$, and that $X^k \to X^*$. Notice that $\Gc$ maps the (compact) interval $[2(I + A^{-1})^{-1}, \half(A+I)]$ to itself. Thus, $\alpha \le \half\norm{I+A}$; since $\frac\alpha{c+\alpha}$ is increasing, we can easily upper-bound $\gamma_1,\gamma_2$ to finally obtain
  \begin{equation*}
    \gamma = \frac{1 + \norm{A}}{1 + \norm{A} + 2\min(\lambda_{\min}(A), 1)}.
  \end{equation*}
  This is strictly smaller than $1$ if $A \succ 0$, thus yielding an explicit contraction rate.
\end{proof}

\noindent\emph{Remark 1:} The bound on $\gamma$ above is a worst-case bound. Empirically, the value of $\alpha$ in the bound is usually much smaller and the convergence rate commensurately faster.

\noindent\emph{Remark 2:} Starting with $X_0=\half(A+I)$ ensures that $X_0 \succ 0$; thus, we can use the iteration~\eqref{eq:8} even if $A$ is semidefinite, as all intermediate iterates remain well-defined. This suggests why iteration~\eqref{eq:8} is empirically robust against ill-conditioning. 

\begin{example}
  \label{eg.lr}
  Suppose $A=0$. Then, the map $\Gc$ is no longer contractive but still nonexpansive. Iterating~\eqref{eq:8} generates the sequence $\{X_k\} = \{\half I, \tfrac38I,\tfrac{33}{112}I,\ldots\}$, which converges to zero.
\end{example}

Example~\ref{eg.lr} shows that iteration~\eqref{eq:8} remains well-defined even for the zero matrix. This prompts us to take a closer look at computing square roots of low-rank semidefinite matrices. In particular, we extend the definition of the S-Divergence to low-rank matrices. Let $A$ be a rank-$r$ semidefinite matrix of size $n\times n$ where $r \le n$. Then, define $\det_r(A) := \prod_{i=1}^r \lambda_i(A)$ to be product of its $r$ positive eigenvalues. Using this, we can extend~(\ref{eq:1}) as
\begin{equation}
  \label{eq:14}
  \delta_{S,r}^2(X,Y) := \log\mydet_r\pfrac{X+Y}{2} - \half\log\mydet_r(X)-\half\log\mydet_r(Y),
\end{equation}
for rank-$r$ SPD matrices $X$ and $Y$. If $\rank(X) \not= \rank(Y)$, we set $\delta_{S,r}(X,Y)=+\infty$. The above definition can also be obtained as a limiting form of~\eqref{eq:1} by applying $\delta_S^2$ to rank deficient $X$ and $Y$ by considering $X+\epsilon I$ and $Y+\epsilon I$ and letting $\epsilon \to 0$.

Although iteration~\eqref{eq:8} works remarkably well in practice, it is a question of future interest on how to obtain square roots for rank deficient matrices faster by exploiting~\eqref{eq:14} instead.


\section{Numerical results}

\label{sec:expts}
We present numerical results comparing running times and accuracy attained by: (i) \algo; (ii) \textsc{Gd}; (iii) \textsc{LsGd}; and (iv) \pn. These methods respectively refer to iteration~\eqref{eq:8}, the fixed step-size gradient-descent procedure of~\citep{jain2015}, our line-search based implementation of gradient-descent, and the polar-Newton iteration of~\citep[Alg.~6.21]{higham08}.

We present only one experiment with \textsc{Gd}, because with fixed steps it vastly underperforms all other methods. In particular, if we set the step size according to the theoretical bounds of~\citep{jain2015}, then for matrices with large condition numbers the step size becomes smaller than machine precision! In our experiments with \textsc{Gd} we used step sizes much larger than theoretical ones, otherwise the method makes practically no progress at all. More realistically, if we wish to use gradient-descent we must employ line-search. Ultimately, although substantially superior to plain \textsc{Gd}, even \textsc{LsGd} turns out to be outperformed by \algo, which in turn is superseded by \pn.

We experiment with the following matrices:
\begin{enumerate}
\vspace*{-5pt}
\setlength{\itemsep}{0pt}
\item $I+ \beta UU'$ for a low-rank matrix $U$ and a variable constant $\beta$. These matrices are well-conditioned.
\item Random Correlation matrices (\textsc{Matlab}: \verb+gallery('randcorr', n)+); medium conditioned.
\item The Hilbert matrix (\textsc{Matlab}: \verb+hilb(n)+). This is a well-known ill-conditioned matrix class.
\item The inverse Hilbert matrix (\textsc{Matlab}: \verb+invhilb(n)+). The entries of the inverse Hilbert matrix are very large integers. Extremely ill-conditioned.
\end{enumerate}

\begin{figure}[h]
  \centering
  \begin{tabular}{ccc}
    \includegraphics[width=0.3\linewidth]{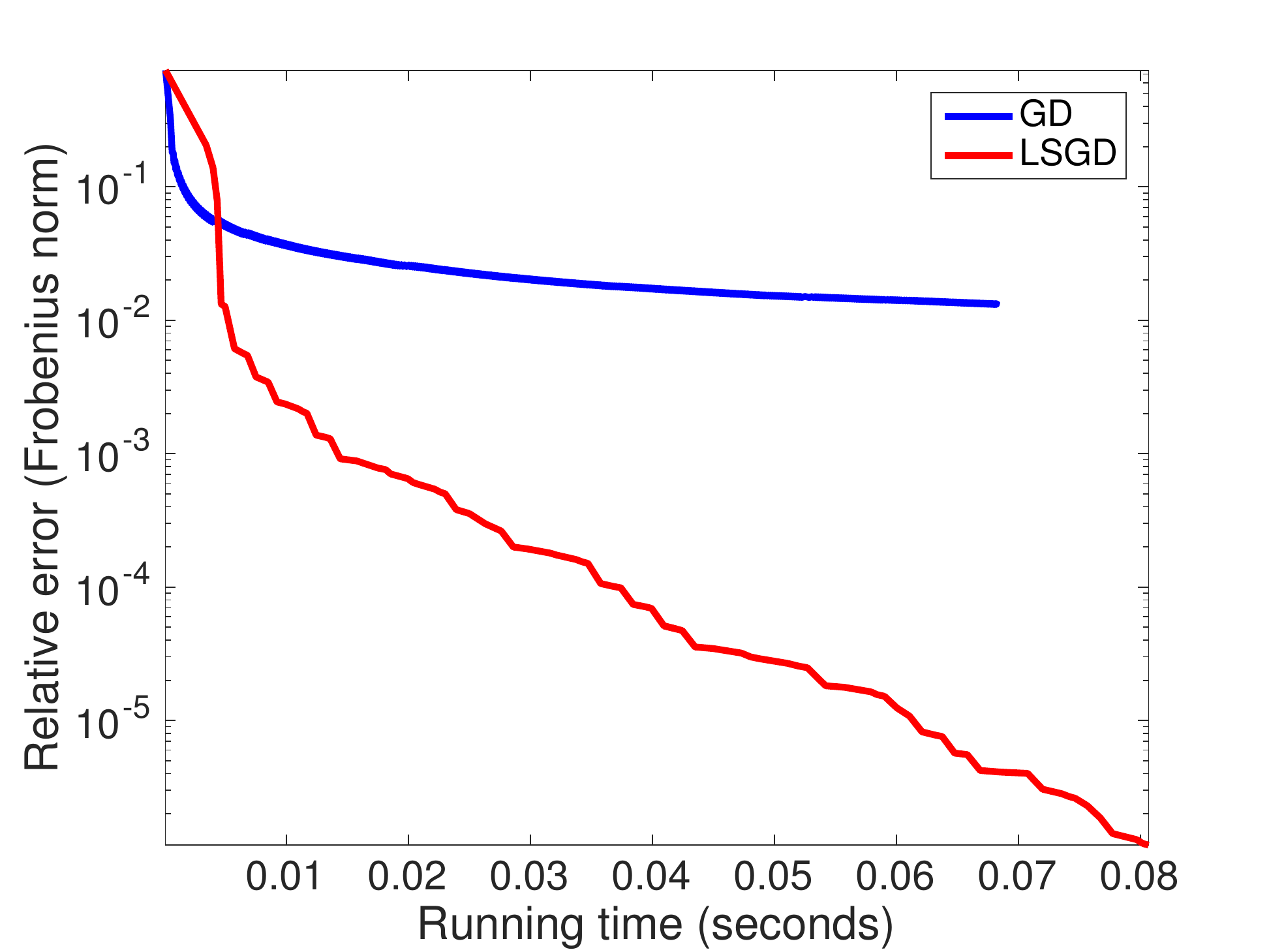}&
    \includegraphics[width=0.3\linewidth]{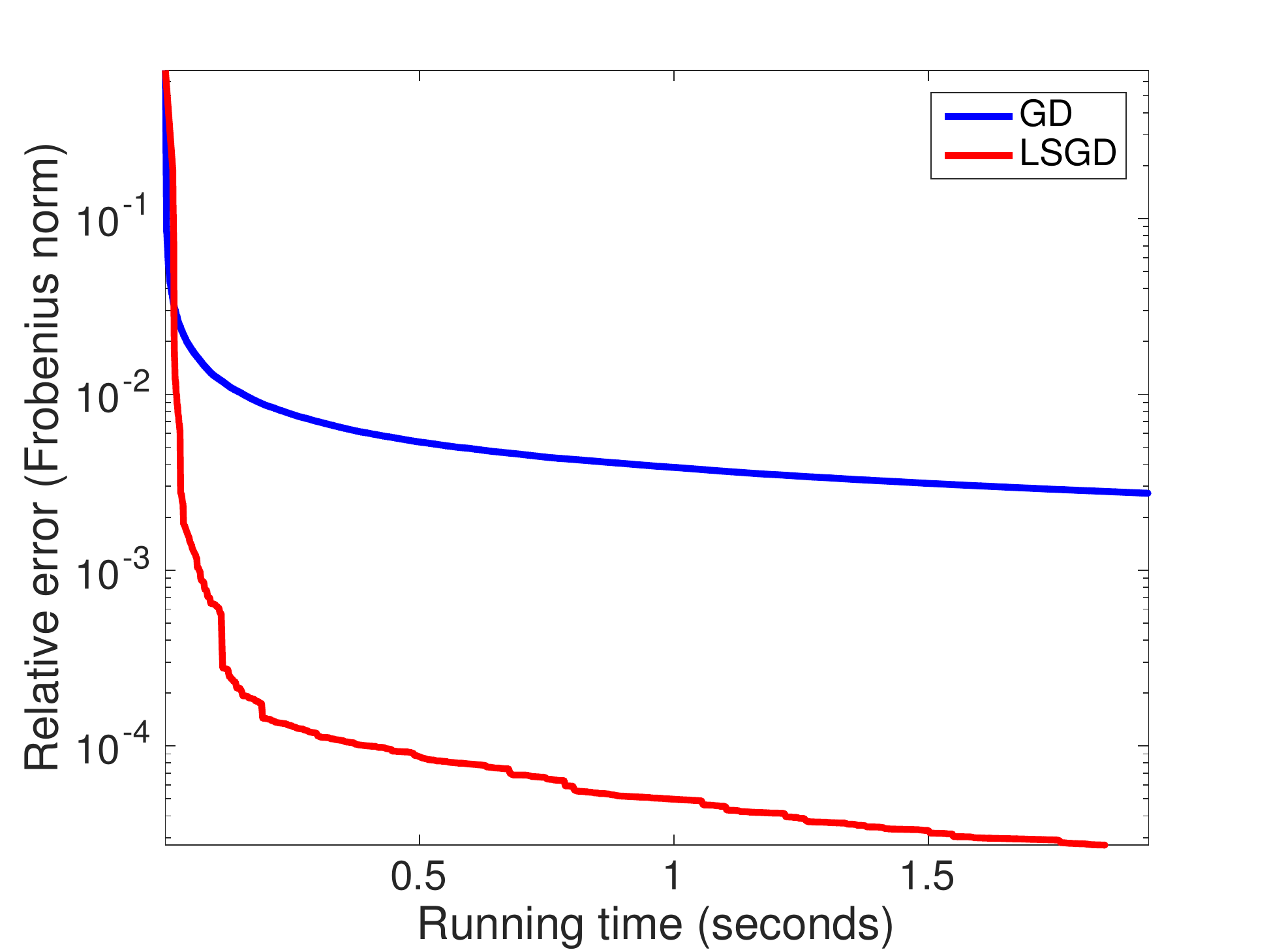}&
    \includegraphics[width=0.3\linewidth]{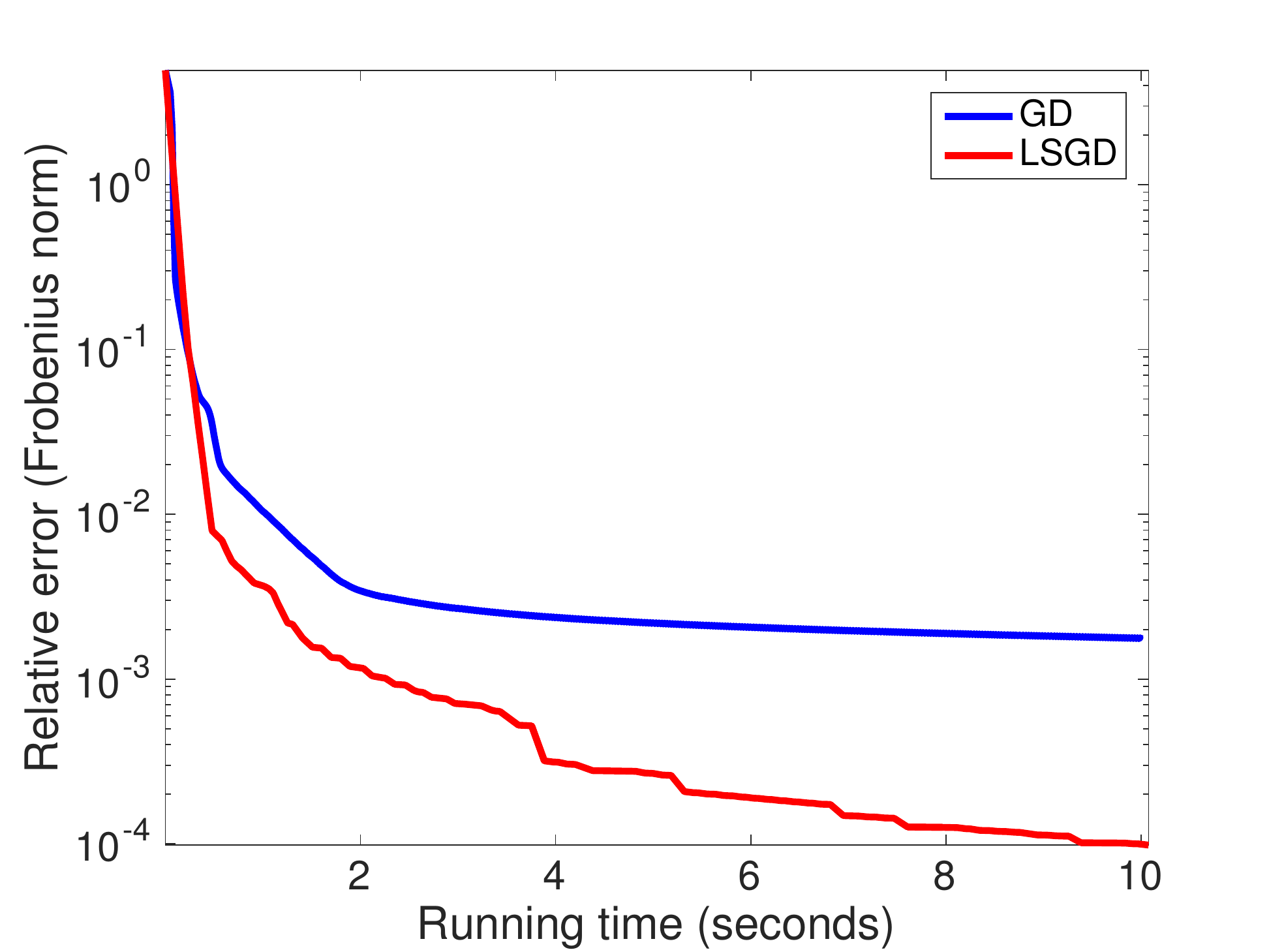}\\
    \footnotesize $50\times 50$ matrix $I+\beta UU^T$& \footnotesize $50\times 50$ Hilbert matrix & \footnotesize $500\times 500$ Covariance
  \end{tabular}
  \caption{Running time comparison between \textsc{Gd} and \textsc{LsGd}; this behavior is typical.}
  \label{fig:lsgd}
\end{figure}

\begin{figure}[h]
  \begin{tabular}{cccc}
    \hskip-12pt\includegraphics[width=0.25\linewidth]{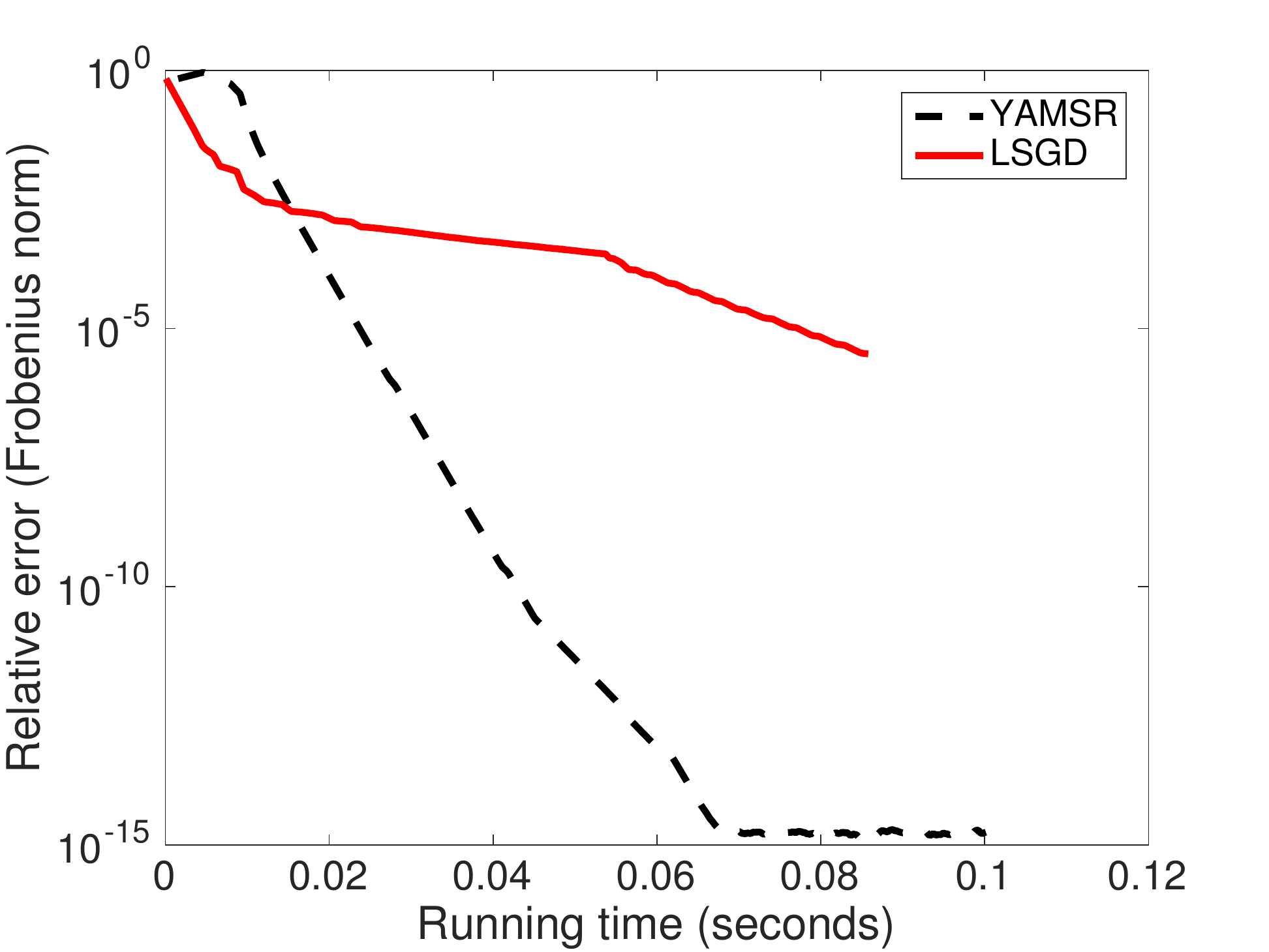}&
    \hskip-12pt\includegraphics[width=0.25\linewidth]{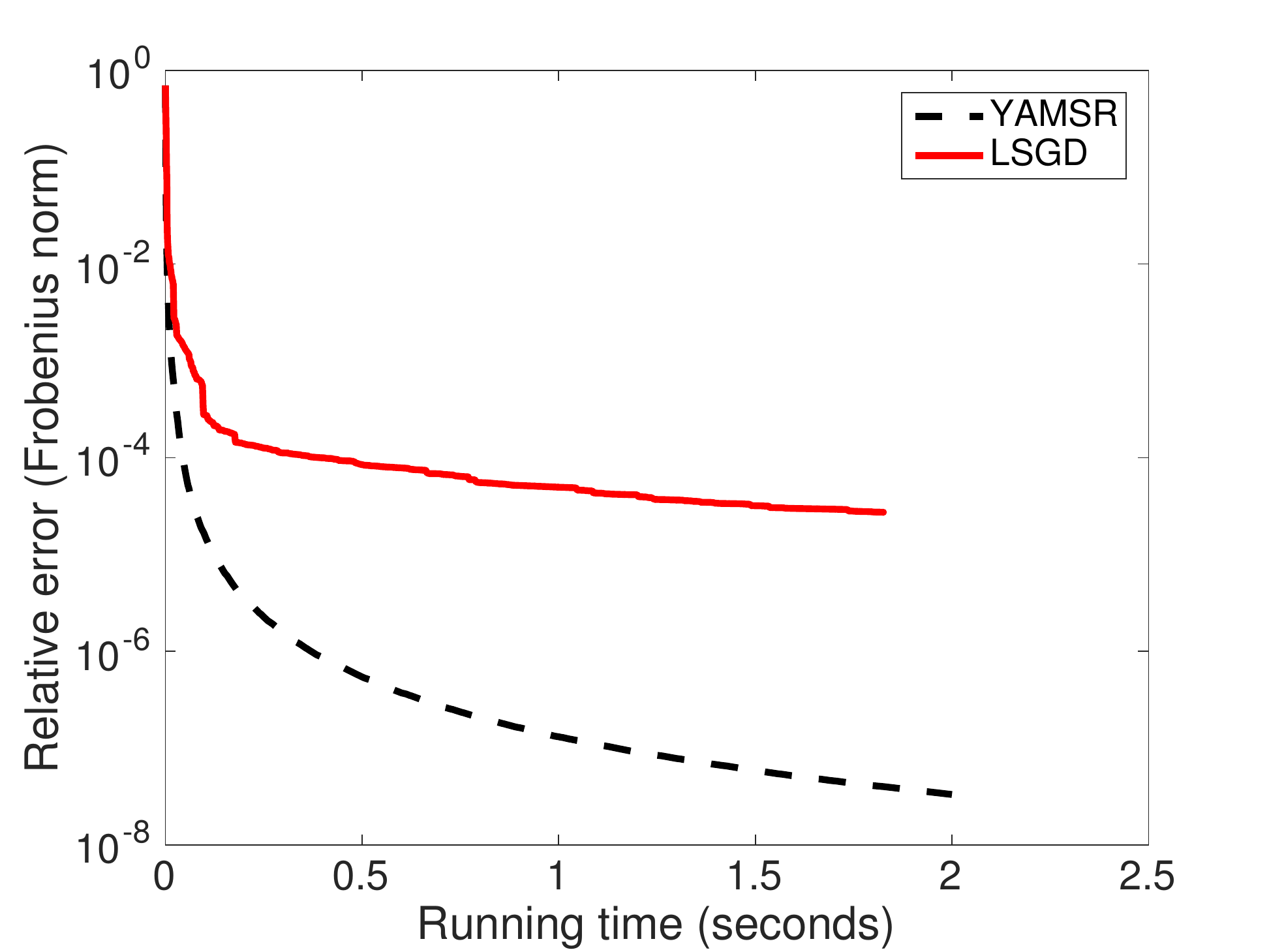}&
    \hskip-12pt\includegraphics[width=0.25\linewidth]{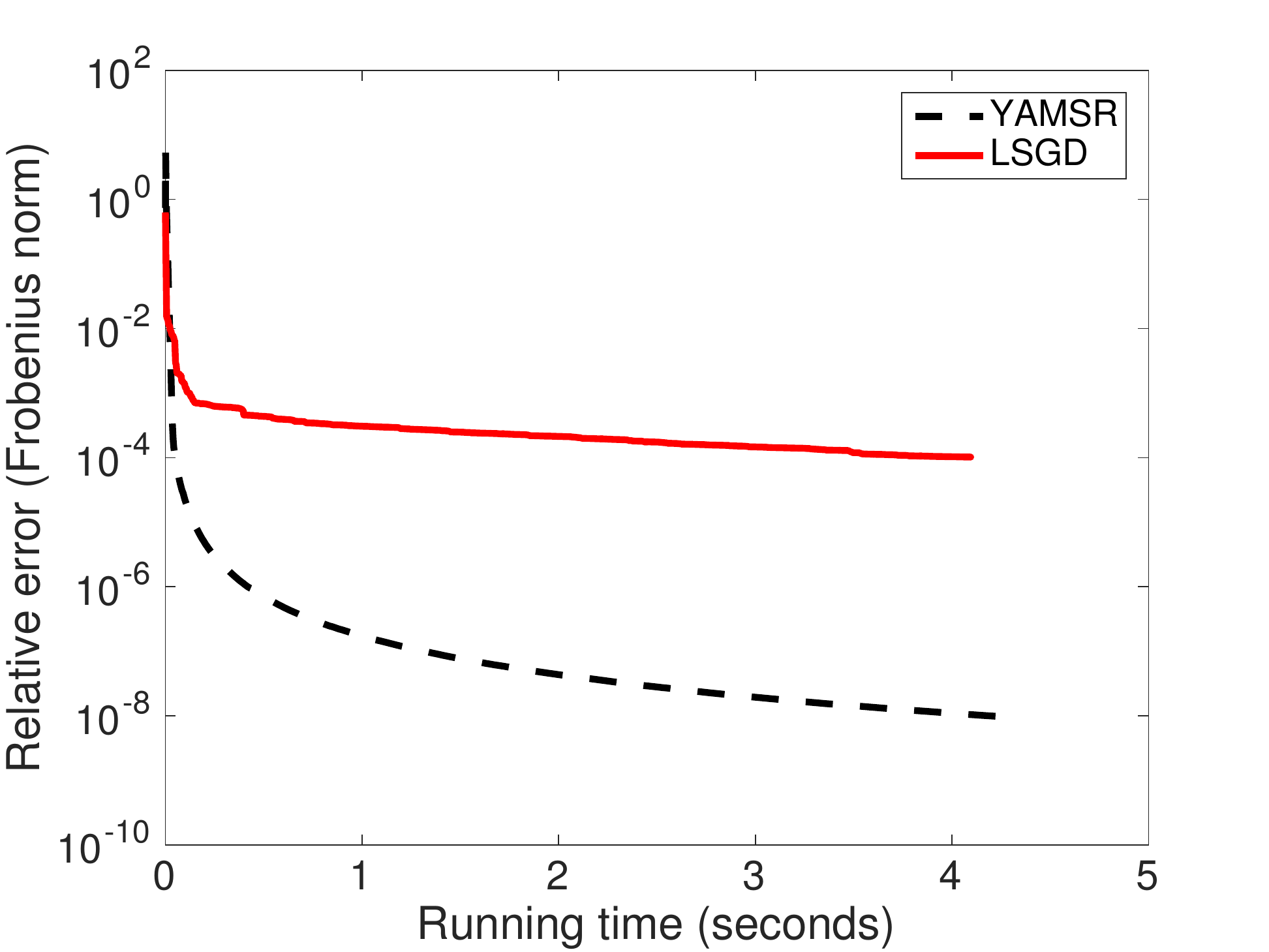}&
    \hskip-12pt\includegraphics[width=0.25\linewidth]{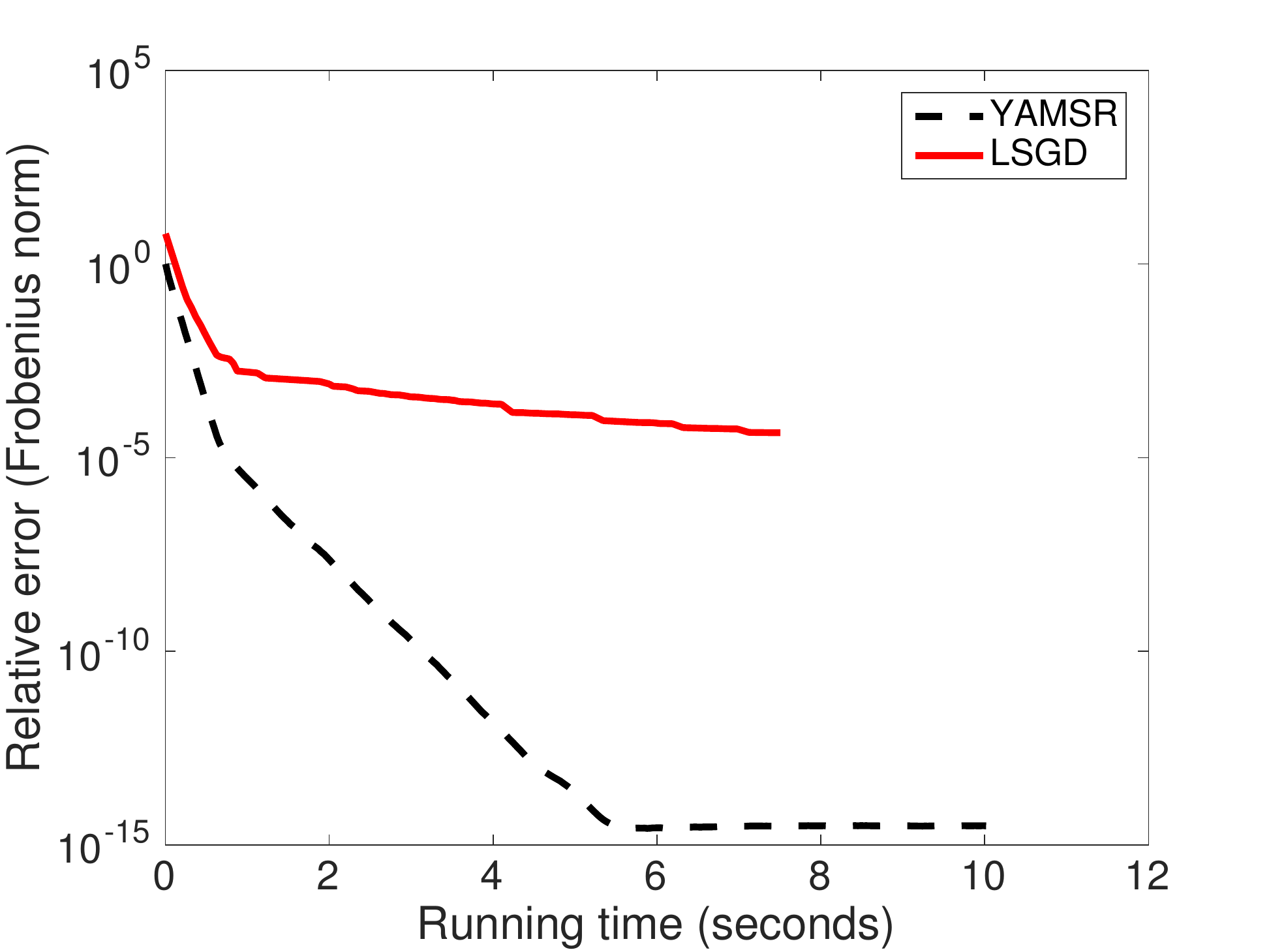}\\
    \footnotesize $50\times 50$ matrix $I+\beta UU^T$&
    \footnotesize $50\times 50$ Hilbert matrix&
    \footnotesize $100\times 100$ inverse Hilbert matrix&
    \footnotesize $500\times 500$ Correlation\\
    $\kappa \approx 64$ & $\kappa \approx 2.8\times 10^{18}$ & $\kappa \approx 9.1\times 10^{96}$ & $\kappa \approx 832$
  \end{tabular}
  \caption{Running time comparison: \algo vs \textsc{LsGd}; this behavior is typical ($\kappa$ is condition number).}
  \label{fig:algo}
\end{figure}

\begin{figure}[!h]
  \begin{tabular}{cccc}
   \hskip-12pt\includegraphics[width=0.25\linewidth]{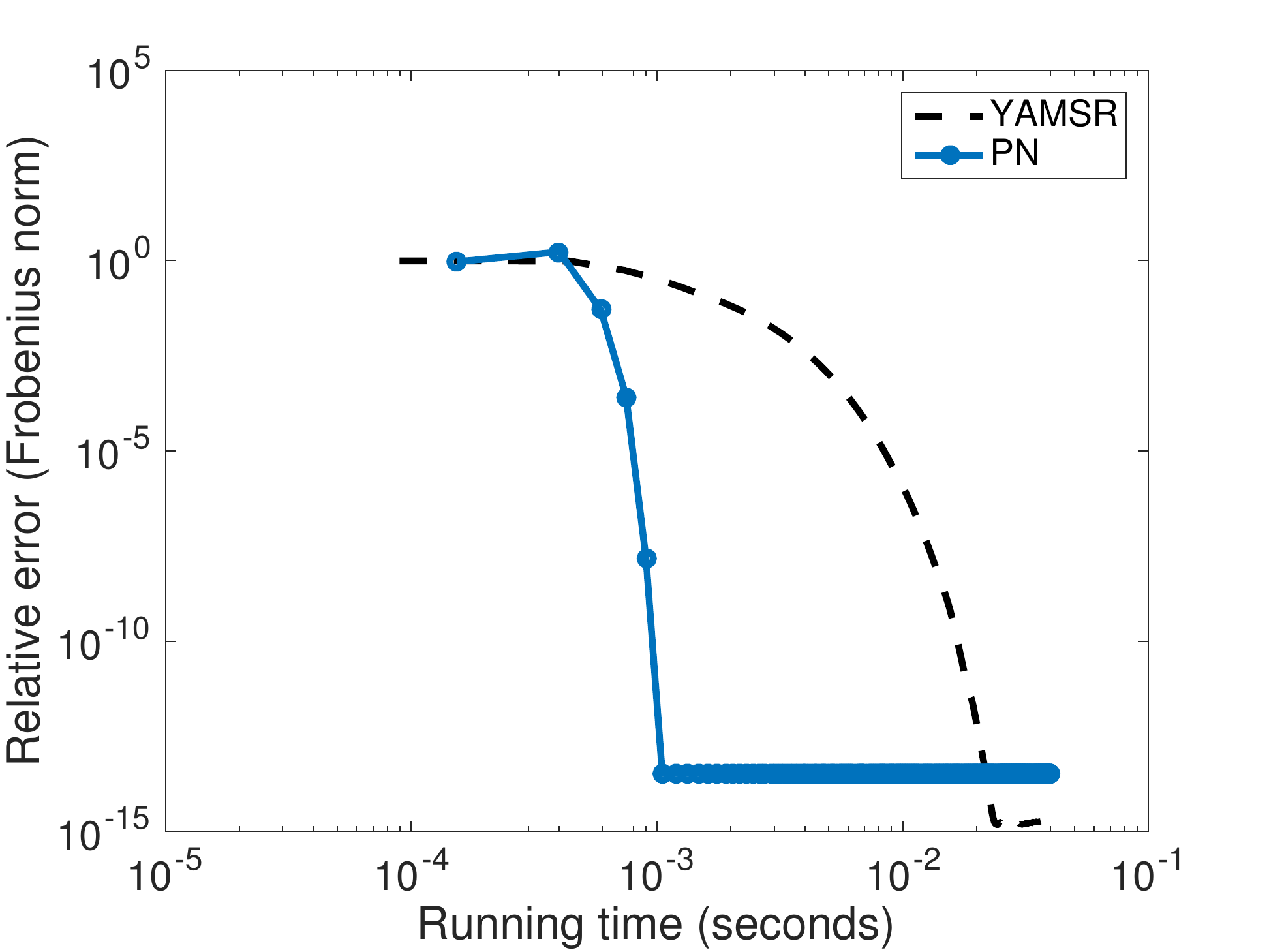}&
   \hskip-12pt\includegraphics[width=0.25\linewidth]{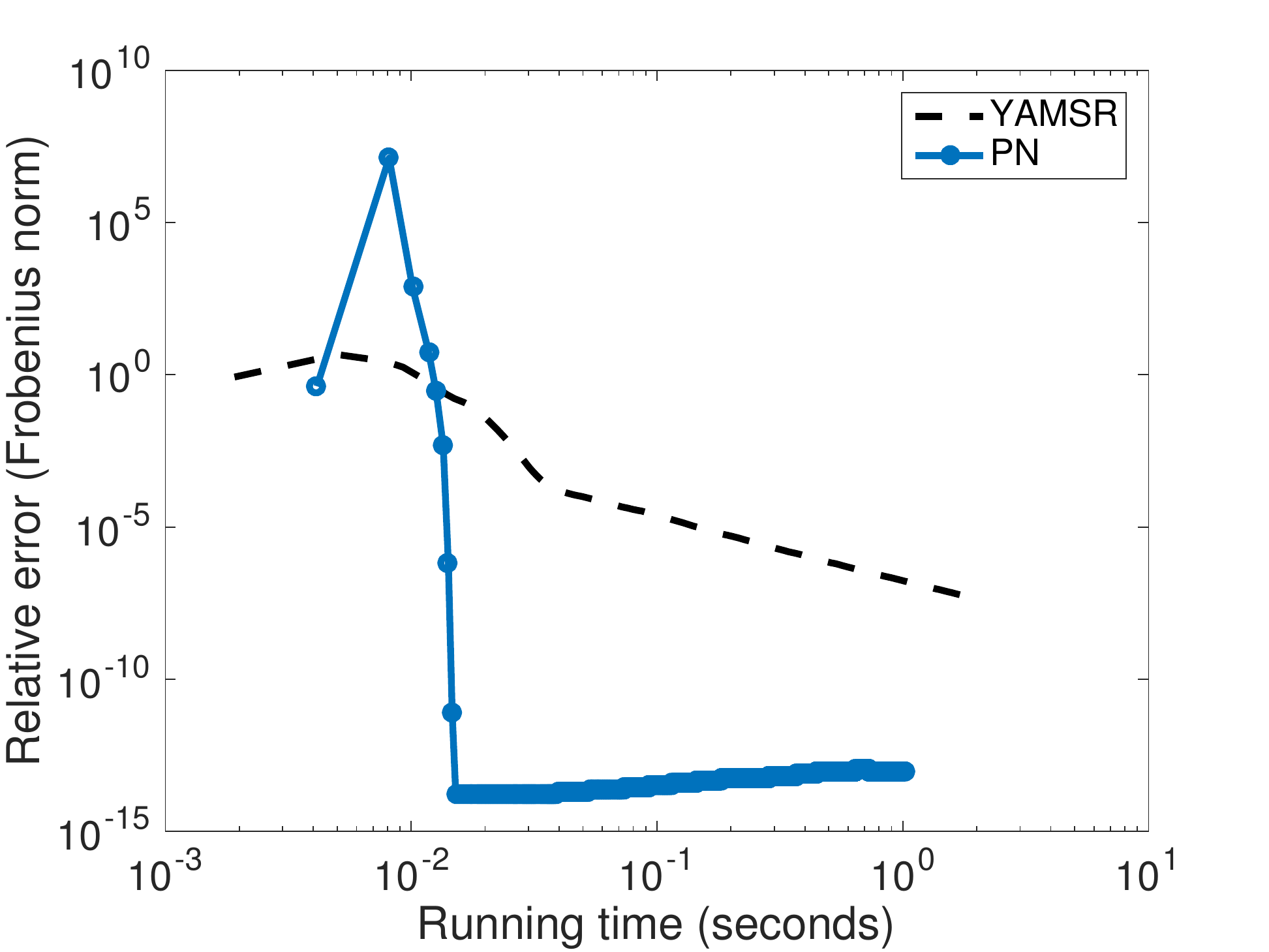}&
   \hskip-12pt\includegraphics[width=0.25\linewidth]{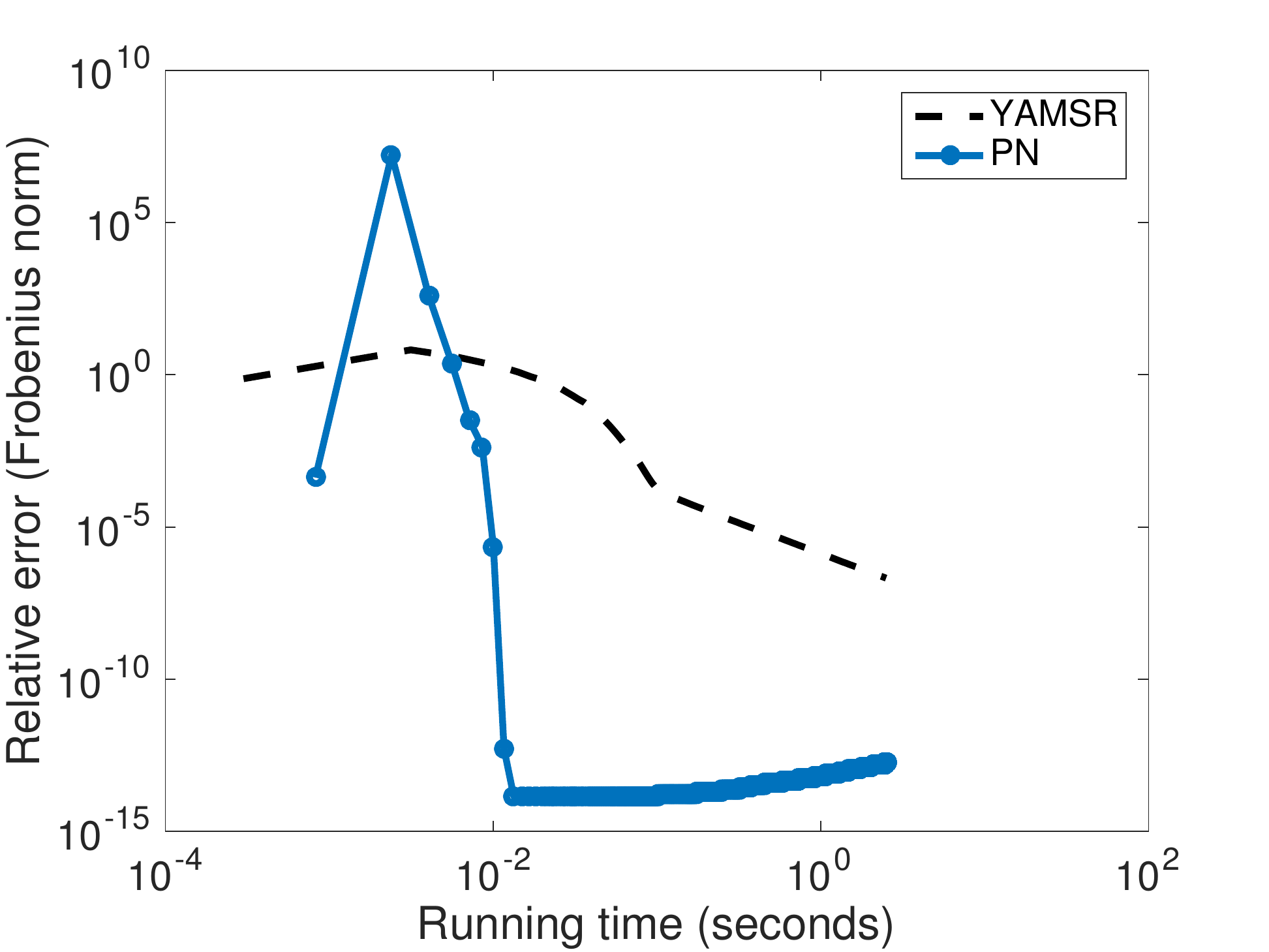}&
   \hskip-12pt\includegraphics[width=0.25\linewidth]{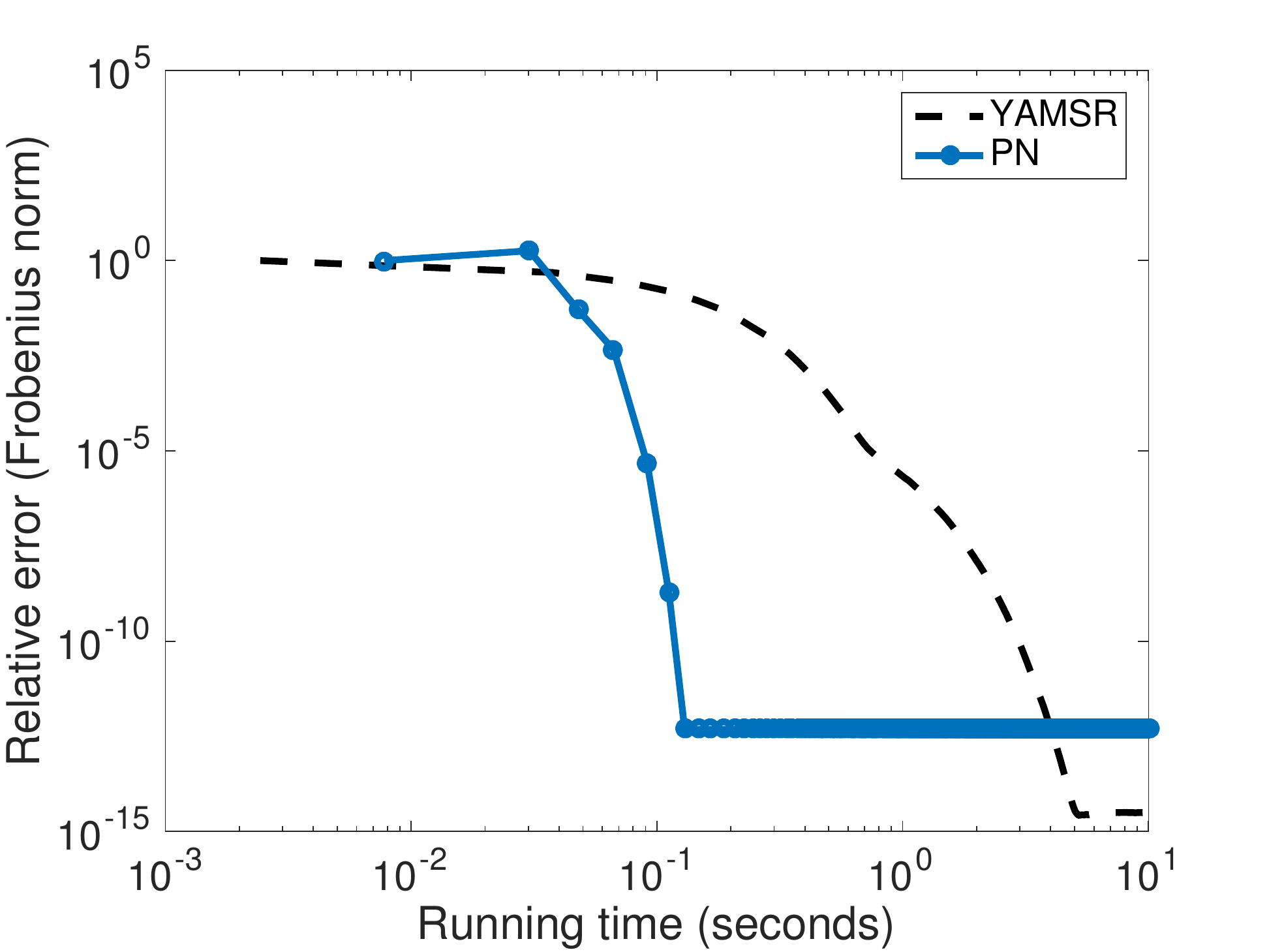}\\
    \footnotesize $50\times 50$ matrix $I+\beta UU^T$&
    \footnotesize $100\times 100$ Hilbert matrix&
    \footnotesize $150\times 150$ inverse Hilbert matrix&
    \footnotesize $500\times 500$ Correlation
  \end{tabular}
  \caption{Running time comparison between \algo and \textsc{Pn}; this behavior is also typical.}
  \label{fig:pn}
\end{figure}

\begin{figure}[h]
  \centering
  \includegraphics[scale=.17]{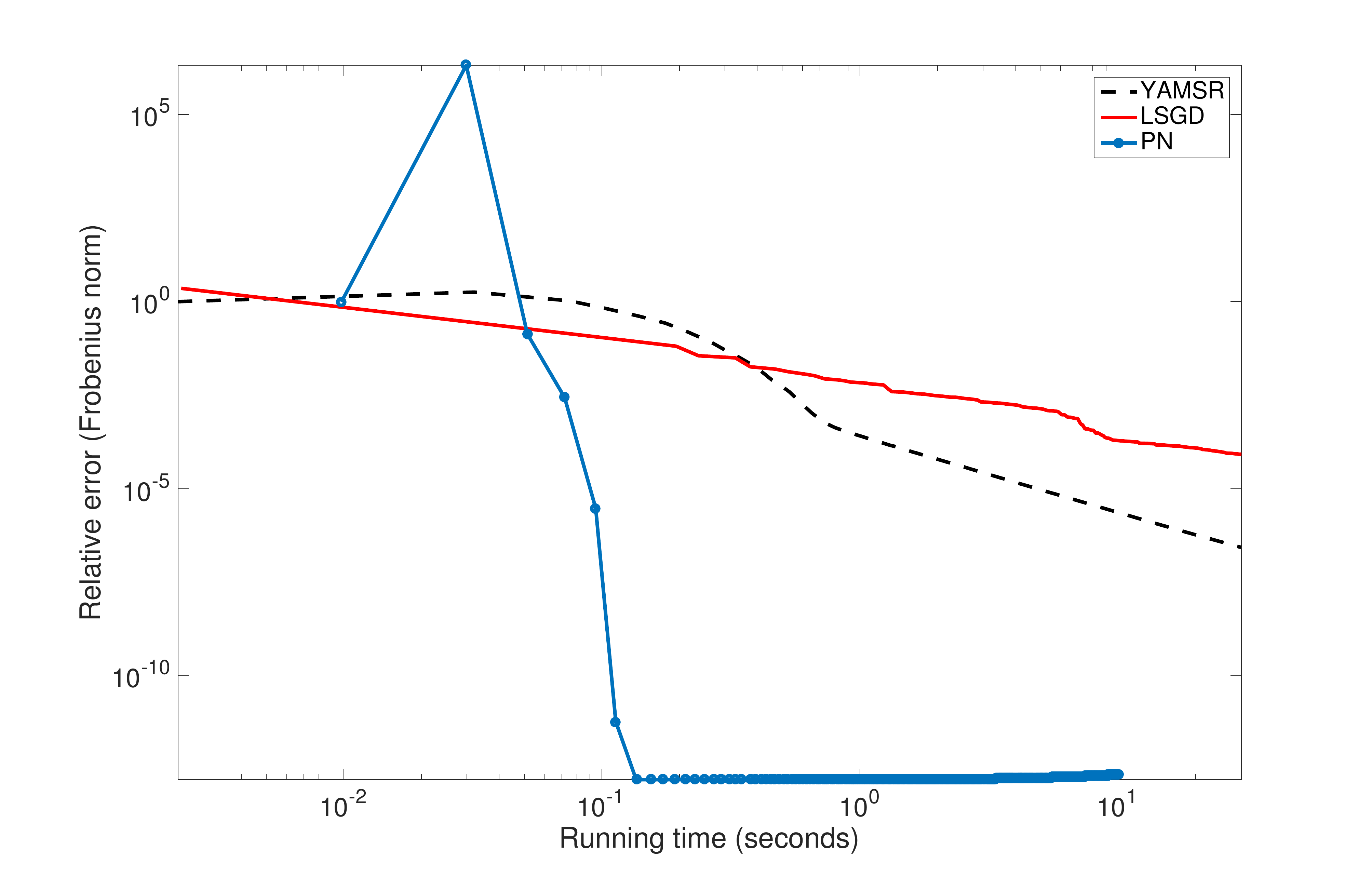}
  \caption{Square root computation of a $500\times 500$ low-rank (rank 50) covariance matrix.}
  \label{fig:lrank}
\end{figure}

\vspace*{-5pt}
\section{Conclusions}
\vspace*{-5pt}
We revisited computation of the matrix square root, and compared the recent gradient-descent procedure of~\citep{jain2015} against the standard polar-Newton method of~\citep[Alg.~6.21]{higham08}, as well as \algo, a new first-order fixed-point algorithm that we derived using the viewpoint of geometric optimization. The experimental results show that the polar-Newton method is the clear winner for computing matrix square roots (except near the $10^{-15}$ accuracy level for well-conditioned matrices). Among the first-order methods \algo outperforms both gradient-descent as well its superior  line-search variant across all tested settings, including singular matrices.

\vspace*{-5pt}
\bibliographystyle{abbrvnat}
{\small\setlength{\bibsep}{3pt}

\begin{thebibliography}{26}
\providecommand{\natexlab}[1]{#1}
\providecommand{\url}[1]{\texttt{#1}}
\expandafter\ifx\csname urlstyle\endcsname\relax
  \providecommand{\doi}[1]{doi: #1}\else
  \providecommand{\doi}{doi: \begingroup \urlstyle{rm}\Url}\fi

\bibitem[Anderson et~al.(1983)Anderson, Morley, and Trapp]{anMoTr83}
J.~Anderson, W.N., T.~Morley, and G.~Trapp.
\newblock Ladder networks, fixpoints, and the geometric mean.
\newblock \emph{Circuits, Systems and Signal Processing}, 2\penalty0
  (3):\penalty0 259--268, 1983.

\bibitem[Ando(1981)]{ando81}
T.~Ando.
\newblock Fixed points of certain maps on positive semidefinite operators.
\newblock In P.~Butzer, B.~Sz.-Nagy, and E.~Görlich, editors, \emph{Functional
  Analysis and Approximation}, volume~60, pages 29--38. Birkhäuser Basel,
  1981.

\bibitem[Ando et~al.(2004)Ando, Li, and Mathias]{andoLiMa04}
T.~Ando, C.-K. Li, and R.~Mathias.
\newblock Geometric means.
\newblock \emph{Linear Algebra and its Applications (LAA)}, 385:\penalty0
  305--334, 2004.

\bibitem[Bhatia(2007)]{bhatia07}
R.~Bhatia.
\newblock \emph{{Positive Definite Matrices}}.
\newblock Princeton University Press, 2007.

\bibitem[Bini and Iannazzo(2011)]{mmtoolbox}
D.~A. Bini and B.~Iannazzo.
\newblock Computing the {K}archer mean of symmetric positive definite matrices.
\newblock \emph{Lin. Alg. Appl.}, Oct. 2011.

\bibitem[Cheng et~al.(2015)Cheng, Cheng, Liu, Peng, and Teng]{chengCheng}
D.~Cheng, Y.~Cheng, Y.~Liu, R.~Peng, and S.-H. Teng.
\newblock {Efficient Sampling for Gaussian Graphical Models via Spectral
  Sparsification}.
\newblock In \emph{Conference on Learning Theory}, pages 364--390, 2015.

\bibitem[Cherian and Sra(2014)]{cherianSra14}
A.~Cherian and S.~Sra.
\newblock {Riemannian Dictionary Learning and Sparse Coding for Positive
  Definite Matrices}.
\newblock \emph{IEEE Transactions Pattern Analysis and Machine Intelligence},
  2014.
\newblock {\it Submitted}.

\bibitem[Hale et~al.(2008)Hale, Higham, and Trefethen]{hale08}
N.~Hale, N.~J. Higham, and L.~N. Trefethen.
\newblock Computing $a^\alpha$, $\log(a)$, and related matrix functions by
  contour integrals.
\newblock \emph{SIAM Journal on Numerical Analysis}, 46\penalty0 (5):\penalty0
  2505--2523, 2008.

\bibitem[Higham(2008)]{higham08}
N.~Higham.
\newblock \emph{{Functions of Matrices: Theory and Computation}}.
\newblock SIAM, 2008.

\bibitem[Higham(1986)]{high86}
N.~J. Higham.
\newblock Newton’s method for the matrix square root.
\newblock \emph{Mathematics of Computation}, 46\penalty0 (174):\penalty0
  537--549, 1986.

\bibitem[Hosseini and Sra(2015)]{hoSra15b}
R.~Hosseini and S.~Sra.
\newblock Matrix manifold optimization for {G}aussian mixture models.
\newblock In \emph{Advances in Neural Information Processing Systems (NIPS)},
  Dec. 2015.

\bibitem[Iannazzo(2011)]{ianna11}
B.~Iannazzo.
\newblock The geometric mean of two matrices from a computational viewpoint.
\newblock \emph{arXiv:1201.0101}, 2011.

\bibitem[Iannazzo and Meini(2011)]{ianMe}
B.~Iannazzo and B.~Meini.
\newblock Palindromic matrix polynomials, matrix functions and integral
  representations.
\newblock \emph{Linear Algebra and its Applications}, 434\penalty0
  (1):\penalty0 174--184, 2011.

\bibitem[Jain et~al.(2015)Jain, Jin, Kakade, and Netrapalli]{jain2015}
P.~Jain, C.~Jin, S.~M. Kakade, and P.~Netrapalli.
\newblock {Computing Matrix Squareroot via Non Convex Local Search}.
\newblock \emph{arXiv:1507.05854}, 2015.
\newblock URL \url{http://arxiv.org/abs/1507.05854}.

\bibitem[Jeuris et~al.(2012)Jeuris, Vandebril, and Vandereycken]{jeuVaVa}
B.~Jeuris, R.~Vandebril, and B.~Vandereycken.
\newblock A survey and comparison of contemporary algorithms for computing the
  matrix geometric mean.
\newblock \emph{Electronic Transactions on Numerical Analysis}, 39:\penalty0
  379--402, 2012.

\bibitem[Kubo and Ando(1980)]{kuboAndo80}
F.~Kubo and T.~Ando.
\newblock Means of positive linear operators.
\newblock \emph{Mathematische Annalen}, 246:\penalty0 205--224, 1980.

\bibitem[Lawson and Lim(2008)]{lawlim08}
J.~Lawson and Y.~Lim.
\newblock A general framework for extending means to higher orders.
\newblock \emph{Colloq. Math.}, 113:\penalty0 191--221, 2008.
\newblock (arXiv:math/0612293).

\bibitem[Lee and Lim(2008)]{leeLim}
H.~Lee and Y.~Lim.
\newblock Invariant metrics, contractions and nonlinear matrix equations.
\newblock \emph{Nonlinearity}, 21:\penalty0 857--878, 2008.

\bibitem[Nakatsukasa and Higham(2012)]{naHi12}
Y.~Nakatsukasa and N.~J. Higham.
\newblock Backward stability of iterations for computing the polar
  decomposition.
\newblock \emph{SIAM Journal on Matrix Analysis and Applications}, 33\penalty0
  (2):\penalty0 460--479, 2012.

\bibitem[Nakatsukasa et~al.(2010)Nakatsukasa, Bai, and Gygi]{yuji}
Y.~Nakatsukasa, Z.~Bai, and F.~Gygi.
\newblock {Optimizing Halley's Iteration for Computing the Matrix Polar
  Decomposition}.
\newblock \emph{SIAM Journal on Matrix Analysis and Applications}, 31\penalty0
  (5):\penalty0 2700--2720, 2010.

\bibitem[Nielsen and Bhatia(2013)]{nieBha13}
F.~Nielsen and R.~Bhatia, editors.
\newblock \emph{{Matrix Information Geometry}}.
\newblock Springer, 2013.

\bibitem[Sra(2015)]{ssdiv}
S.~Sra.
\newblock {Positive Definite Matrices and the S-Divergence}.
\newblock \emph{Proceedings of the American Mathematical Society}, 2015.
\newblock arXiv:1110.1773v4.

\bibitem[Sra and Hosseini(2015)]{sraHo15}
S.~Sra and R.~Hosseini.
\newblock {Conic Geometric Optimization on the Manifold of Positive Definite
  Matrices}.
\newblock \emph{SIAM J. Optimization (SIOPT)}, 25\penalty0 (1):\penalty0
  713--739, 2015.

\bibitem[Wiesel(2012{\natexlab{a}})]{wie12}
A.~Wiesel.
\newblock Geodesic convexity and covariance estimation.
\newblock \emph{{IEEE} Transactions on Signal Processing}, 60\penalty0
  (12):\penalty0 6182--89, 2012{\natexlab{a}}.

\bibitem[Wiesel(2012{\natexlab{b}})]{wie12b}
A.~Wiesel.
\newblock Unified framework to regularized covariance estimation in scaled
  {G}aussian models.
\newblock \emph{{IEEE} Transactions on Signal Processing}, 60\penalty0
  (1):\penalty0 29--38, 2012{\natexlab{b}}.

\bibitem[Zhang(2012)]{zhang12}
T.~Zhang.
\newblock Robust subspace recovery by geodesically convex optimization.
\newblock \emph{arXiv preprint arXiv:1206.1386}, 2012.

\end{thebibliography}

}

\appendix
\vspace*{-5pt}
\section{Technical details}
\vspace*{-5pt}
\begin{proof}[Proof of Theorem~\ref{thm:gc}] 
  It suffices to prove midpoint convexity; the general case follows by continuity. Consider therefore psd matrices $X_1, X_2, Y_1, Y_2$. We need to show that
  \begin{equation}
    \label{eq.sgc}
    \delta_S^2(X_1\gm_{1/2} X_2, Y_1\gm_{1/2} Y_2)
    \le \half\delta_S^2(X_1,Y_1) + \half\delta_S^2(X_2,Y_2).
  \end{equation}
  From the joint concavity of the operator $\gm_{1/2}$~\citep{kuboAndo80} we know that
  \begin{equation*}
    X_1 \gm_{1/2} X_1 + Y_1 \gm_{1/2} Y_2 \preceq (X_1+Y_1) \gm_{1/2} (X_2+Y_2).
  \end{equation*}
  Since $\log\det$ is a monotonic function, applying it to this inequality yields
  \begin{equation}
    \label{eq:9}
    \log\det[X_1 \gm_{1/2}  X_1 + Y_1 \gm_{1/2} Y_2]
    \le
    \log\det[(X_1+Y_1) \gm_{1/2} (X_2+Y_2)].
  \end{equation}
  But we also know that $\log\det((1-t)X\gm_{t} t Y)=(1-t)\log\det(X)+t\log\det(Y)$. Thus, a brief algebraic manipulation of~\eqref{eq:9} combined with the definition~(\ref{eq:1}) yields inequality~\eqref{eq.sgc} as desired. 
\end{proof}

\end{document}